\documentclass[12pt]{article}

\usepackage{latexsym,amssymb,amsmath,amsthm,euscript,graphicx}
\usepackage{arydshln}

\newcommand{\DEF}[1]{\textit{#1}}
\newcommand{\dsp}{\displaystyle}

\newtheorem{thm}{Theorem}[section]

\newtheorem{lem}[thm]{Lemma}
\newtheorem{prp}[thm]{Proposition}

\newtheorem{alg}[thm]{Algorithm}

\newcommand{\mapsfrom}{\mathrel{\reflectbox{\ensuremath{\mapsto}}}}

\begin{document}

\setlength{\baselineskip}{1.25 \baselineskip}

\def \today {\number \day \ \ifcase \month \or January\or February\or
  March\or April\or May\or June\or July\or August\or
  September\or October\or November\or December\fi\
  \number \year}


\title{\sffamily On lengths of burn-off chip-firing games} 

\author{
   \textsc{P. Mark Kayll\footnotemark} \\[0.25em]
    {\small\textit{Department of Mathematical Sciences}} \\[-0.25em]
    {\small\textit{University of Montana}} \\[-0.25em]
    {\small\textit{Missoula MT 59812, USA}} \\[-0.1em]
    {\small\texttt{mark.kayll@umontana.edu}}
    \and
    \textsc{Dave Perkins\footnotemark} \\[0.25em]
    {\small\textit{Computer Science Department}} \\[-0.25em]
    {\small\textit{Hamilton College}} \\[-0.25em]
   {\small\textit{Clinton NY 13323, USA}} \\[-0.1em]
    {\small\texttt{dperkins@hamilton.edu}}
   }

\date{\small 02-02-2020}

\maketitle

%
\renewcommand{\thefootnote}{}
\footnotetext{2010 {\em MSC\/}:\ Primary
05C57;   
Secondary
05C85,   
05C05,   
05C25,   
60J20,   
68R10,   
91A43.   
}

%
%
\renewcommand{\thefootnote}{\fnsymbol{footnote}}

\addtocounter{footnote}{1}
\footnotetext{This work was partially supported by a grant from the Simons Foundation (\#279367 to Mark Kayll).} 
\addtocounter{footnote}{1}
\footnotetext{Part of this work appears in the author's PhD dissertation~\cite{perkins2005}.}

\begin{center}
\phantom{In recognition of Gary MacGillivray's milestone birthday in 2020} 
\phantom{No jokes in the paper, but not so in the talk!}
\end{center}

    \begin{abstract}

\noindent
We continue our studies of burn-off chip-firing games from
[\textit{Discrete Math.\ Theor.\ Comput.\ Sci.} \textbf{15} (2013), no.~1, 
  121--132; MR3040546]
and
[\textit{Australas.\ J.\ Combin.} \textbf{68} (2017), no.~3, 
  330--345; MR3656659].
The latter article introduced randomness by
choosing successive seeds uniformly from the vertex set of a
graph $G$. The length of a game is the number of vertices 
that fire (by sending a chip to each neighbor and annihilating 
one chip) as an excited chip
configuration passes to a relaxed state. This article determines
the probability distribution of the game length in a long sequence
of burn-off games. Our main results give exact counts for the 
number of pairs $(C,v)$, with $C$ a relaxed legal configuration
and $v$ a seed, corresponding to each possible length.
In support, we give our own proof of the well-known equicardinality
of the set $\mathcal{R}$ of relaxed legal configurations on 
$G$ and the set of spanning trees in the cone $G^*$ of $G$. 
We present an algorithmic, bijective proof of this correspondence.


   \end{abstract}

\section{Introduction}

This article continues our study in \cite{KayllPerkins2013} and
\cite{PerkinsKayll2016} of burn-off chip-firing games, in 
which each iteration simulates the loss of energy from a 
complex system.  
These games are played on graphs and consist of a sequence of
`seed-then-relax' steps, wherein a chosen vertex is excited (by adding
a `chip' to it) after which the system (i.e.\ a graph containing chips
on its vertices) is allowed to `relax'. During relaxation, certain
vertices `fire' (by sending chips to their neighbors and annihilating
a chip); the `length' of a game is the number of such vertices. We
shall see that the firing order and number of firings at any given
moment has no effect on the eventual relaxation; so, e.g., the notion
of length is well defined (see Lemmas~\ref{lem-2.1} and
\ref{lem-4.8}).  In \cite{PerkinsKayll2016}, we introduced randomness
to these games by choosing each successive seed uniformly at random
from among all possible vertices. The present work aims primarily at
shedding light on the probability distribution of the game length in a
long sequence of burn-off games. Our main results in this
direction---Proposition~\ref{prp-4.7} and Theorem~\ref{thm-4.9}---give
exact counts for the number of pairs $(C,v)$, with $C$ a `relaxed
legal chip configuration' and $v$ a seed vertex, corresponding to each
possible game length.

En route to these results, we (re)discovered that, for a graph $G$, our
set $\mathcal{R}$ of relaxed legal configurations on $G$ is
equicardinal to the set $\mathcal{S}$ of spanning trees in the `cone'
$G^*$ of $G$. We present an algorithmic, bijective proof of this fact
in Section~\ref{Sec-counting-R} (Theorem~\ref{thm-R-is-S}).
The connection between chip firing and spanning tree enumeration has
been addressed by numerous authors (e.g., 
\cite{Baker-Shokrieh2013},
\cite{BCT2010}, \cite{Biggs-Winkler97}, \cite{Biggs99},
\cite{Holroyd-et-al2008}, \cite{KayllPerkins2013}), but we 
present our take for several 
reasons. First, our main results in Section~\ref{Sec-dist-game-lengths}
rest on ideas in our proof in Section~\ref{Sec-counting-R}.
Second, that
\begin{equation}
\label{main-RS}
|\mathcal{R}|=|\mathcal{S}|
\end{equation}
is a key connecting $\mathcal{R}$ with the `sandpile group'
$K(G^*)$; 
thus we recover an appealing description of the elements of this group.
Finally, we hope that our constructive proof stands up,
of interest in its own right.

We attempt neither a literature review nor a discussion of background
or motivation for chip firing. 
Perhaps the most immediate resource for related material is David
Perkinson's beautiful Sandpiles
website~\cite{Perkinson-sandpile-page}, 
which, besides literature links, provides access to simulation
software including Sage tools. 
We also point the reader to our
other papers \cite{KayllPerkins2013}, \cite{PerkinsKayll2016}, 
\cite{PerkinsKayll2017}, to the surveys \cite{Holroyd-et-al2008}, 
\cite{Merino2005}, to the books \cite{CorryPerkinson2018},  \cite{Klivans2019},
and to the concise but thorough 
AMS column \cite{Levine-Propp2010}.

The rest of this article is organized as
follows. First (in Section~\ref{sect-description}), we introduce the basic
chip-firing notions, including the undefined terms already
encountered. In Section~\ref{sect-sandpile-rmks}, we take a brief detour to
explain the connection between $\mathcal{R}$ and $K(G^*)$ 
implied by (\ref{main-RS}). 
Section~\ref{Sec:background_results} details
the earlier lemmas and tools supporting our main results. In
Section~\ref{Sec-counting-R}, we present our 
proof of (\ref{main-RS}). Our main results counting pairs
$(C,v)\in\mathcal{R}\times V$ with specified game lengths appear in
Section~\ref{Sec-dist-game-lengths}. In Section~\ref{Sec:example}, we
close with an example illustrating the use of
Theorem~\ref{thm-R-is-S}, Proposition~\ref{prp-4.7}, and
Theorem~\ref{thm-4.9} in determining the probability distribution for
game length. 

\subsubsection*{Notation and terminology}

In this paper, all graphs are finite, simple, and undirected.  We
usually think of playing burn-off games on connected graphs, but
most of our results don't require connectivity; cf.\ the first
paragraph in the proof of Theorem~\ref{thm-R-is-S}. We use 
`general graph' when we wish to emphasize that a graph 
may be disconnected. The order of a graph $G=(V,E)$ is denoted 
by $n$ ($:=|V|$).  If $G$ has a subgraph $X$ and $v\in V(G)$, then
$\Gamma_X(v)$ denotes the set of neighbors of $v$ that lie in $V(X)$.
If $G$ is connected and $u,v\in V$, then the least length of a
$uv$-path in $G$ is the \DEF{distance} $d_G(u,v)$ from $u$ to
$v$. Finally, we write $\tau=\tau(G)$ for the number of spanning trees
of $G$. 

We mainly follow usual graph theory conventions as found, e.g., in
\cite{BondyMurty08} and refer the reader there for any 
omitted items of this sort. A graph theory reference that
addresses chip firing specifically is \cite{GodsilRoyle2001}. For
probability background, see the classic~\cite{Feller-vI-68}.

\subsection{Burn-off chip firing}
\label{sect-description}

Beginning with a \DEF{(chip) configuration} on a graph
$G=(V,E)$---i.e., a function $C\colon V\to\mathbb{N}$---a
\DEF{burn-off (chip-firing) game} plays as follows.
For a vertex $v$, if $C(v)$ exceeds $\deg_{G}(v)$, then $v$ can
\DEF{fire}, meaning it sends one chip to each neighbor and one chip
into `thin air'. Formally, when $v$ fires, $C$ is
modified to a configuration $C'$ such that
\begin{equation}
\label{fire-two}
C'(u) = \left\{\begin{array}{ll}
    C(v)-\deg_G(v)-1 & \text{if } u=v, \\
    C(u)+1 & \text{if } uv\in E(G), \\
    C(u) & \text{if } v\neq u\not\sim v.
    \end{array}\right.
\end{equation}
As we noted in \cite{KayllPerkins2013}, the game just defined is
equivalent to the `dollar game' of Biggs~\cite{Biggs99} in the
case when his `government' vertex is adjacent to every other
vertex in the underlying graph; it is
also equivalent to the sandpile model on $G^*$ (see, e.g.,
\cite{Holroyd-et-al2008}). 

For a configuration $C$, a vertex $v$ is \DEF{critical} if
$C(v)=\deg_G{(v)}$ and \DEF{supercritical} if $C(v)>\deg_G{(v)}$.  A
\DEF{relaxed} configuration is one for which no vertex can fire.  To
start a burn-off game, we add a chip to a selected vertex $v$ (called
a \DEF{seed}) in a relaxed configuration $C$. This is called
\DEF{seeding} $C$ at $v$ and is sometimes denoted algebraically: by
writing $\mathbf{1}_v$ for the configuration with a total of one chip,
on $v$ only, and passing from $C$ to $C+\mathbf{1}_v$. Just prior to
seeding, if $v$ happened to be critical, then from $C+\mathbf{1}_v$,
we fire $v$, which may trigger a neighbor $u$ of $v$ to become
supercritical. If so, we fire $u$, which may trigger another vertex to
become supercritical. The game follows this cascade until reaching a
relaxed configuration, called a \DEF{relaxation} of
$C+\mathbf{1}_v$. The game \DEF{length} equals the number of vertex
firings, possibly zero, in passing from the initial relaxed
configuration to the final one.

In a long game sequence, certain sparse configurations will cease to
appear after enough seedings. Let us suppose, for example, that a game
sequence is initialized with the all-zeros configuration.  Except on a
trivial graph, this configuration will never recur, and a
configuration $\mathbf{1}_v$ on a triangle ($K_3$) also will never
be seen after its first occurence. Loosely speaking, by `legal'
configurations, we mean those typically encountered in a long game
sequence. To define these formally, we begin by calling a
configuration \DEF{supercritical} if every vertex is supercritical.
We follow our earlier papers 
\cite{KayllPerkins2013}, \cite{PerkinsKayll2016},  
and focus on the configurations that can result from relaxing
supercritical ones. First
consider what happens when a burn-off game is played in reverse. 
Considering (\ref{fire-two}), we see that to start in a configuration
$C'$ and \DEF{reverse-fire} a vertex $v$
(each of whose neighbors $u$
necessarily satisfies $C'(u)\geq 1$) means to modify $C'$ to a
configuration $C$ such that
\begin{equation*}
\label{rev-fire-two}
C(u) = \left\{\begin{array}{ll}
    C'(v)+\deg_G(v)+1 & \text{if } u=v, \\
    C'(u)-1 & \text{if } uv\in E(G), \\
    C'(u) & \text{if } v\neq u\not\sim v.
    \end{array}\right.
\end{equation*}
Now a configuration $C$ is \DEF{legal} if there exists a
reverse-firing sequence starting with $C$ and ending with a
supercritical configuration. Throughout this paper, we use
$\mathcal{R}=\mathcal{R}(G)$ to denote the set of relaxed legal
configurations on $G$.

A relaxed configuration $C$ is \DEF{recurrent} if, given any
(unrestricted) configuration $C'$, it is possible to pass from $C'$ to
$C$ via a sequence of seeding vertices and firing supercritical ones.

\subsection{The sandpile group}
\label{sect-sandpile-rmks}

As mentioned following (\ref{main-RS}), the set 
$\mathcal{R}$ is linked to $G^*$'s sandpile group, which
we proceed to define (see Section~\ref{Sec-counting-R}
for a definition of $G^*$ itself). Start by viewing configurations 
$C\colon V\to\mathbb{N}$ as elements of the group
$\mathbb{Z}^{V}$. Looking at (\ref{fire-two}), notice that firing a
vertex $v\in V$ corresponds to adding to $C$ the vector
$\Delta_{v}\in\mathbb{Z}^{V}$ with entries
\begin{equation*}
\Delta_{v,u} := \left\{\begin{array}{cl}
    -\deg_{G^*}(v) & \text{if } u=v, \\
    1 & \text{if } uv\in E(G), \\
    0 & \text{if } v\neq u\not\sim v,
    \end{array}\right.
\end{equation*}
in which $u$ runs through $V$. The matrix
$\Delta:=(\Delta_{v,u})=(\Delta_{u,v})$ is the
\DEF{reduced Laplacian} of $G^*$ (``reduced'' as it omits the row/column
corresponding to the universal vertex introduced in passing 
from $G$ to $G^*$), and thus we see that chip firing provides a
natural setting for the appearance of $\Delta$ (see, e.g.,
\cite{BondyMurty08} for background on the graph Laplacian). 
The idea that configurations appearing in a sequence of
vertex firings enjoy an intimate connection motivates 
calling two configurations $C$, $D$ \DEF{firing equivalent}  exactly
when $C-D$ lies in the $\mathbb{Z}$-linear span $\Delta\mathbb{Z}^{V}$
of the vectors $\Delta_{v}$, i.e., when $C$ and $D$ lie in the same
coset of the quotient group $\mathbb{Z}^{V}/\Delta\mathbb{Z}^{V}$.
This is the \DEF{sandpile group} of $G^*$ and is denoted by
$K(G^*)$. Our discussion here follows \cite{Levine-Propp2010}, which
gives a chockablock introduction to the subject.

Before presenting our own results, we record an observation on 
the role of (\ref{main-RS}) in connecting $\mathcal{R}$ with $K(G^*)$.

\begin{prp}
  \label{R-is-sandpile_grp}
The elements of $\mathcal{R}$  can serve as a set of representatives
for $K(G^*)$.
\end{prp}

\begin{proof}
First note that  both of $\mathcal{R}$, $K(G^*)$ contain $\tau(G^*)$
elements. For $\mathcal{R}$ , this is (\ref{main-RS}) (our
Theorem~\ref{thm-R-is-S}) and for $K(G^*)$, this is also well known (see,
e.g., \cite{Levine-Propp2010}). Furthermore, members of 
$\mathcal{R}$ are all recurrent configurations, a fact we proved in
\cite[Proposition~3.1]{PerkinsKayll2016} 
(though it was known much earlier in the sandpile literature; cf.\,\cite{Dhar90}). 
Now each equivalence class
of $\mathbb{Z}^{V}$ (under the firing equivalence) contains exactly
one recurrent configuration (see \cite{Levine-Propp2010} again, or,
e.g., \cite{Holroyd-et-al2008}). So we have $\tau(G^*)$ recurrent
configurations (in $\mathcal{R}$) and the same number of recurrent
configurations appearing among the elements of $K(G^*)$
(i.e., among the the equivalence classes of
$\mathbb{Z}^{V}$), the latter being exhaustive. Therefore,
$\mathcal{R}$ must be the set of all recurrent configurations.
\end{proof}

Proposition~\ref{R-is-sandpile_grp} is not new. Indeed, it recasts the
modern definition of $K(G^*)$ above in terms of the original
definition due to Dhar~\cite{Dhar90}. Nevertheless, it's striking to 
observe the central role that enumeration plays in its proof.

\section{Supporting results}
\label{Sec:background_results}

In Section~\ref{sect-description}, we glossed over whether the length of a
burn-off game is well defined. The following early chip-firing result
settles this question and shows that the relaxation of a configuration
is uniquely determined. 

\begin{lem}[\cite{Dhar90},\cite{DiaconisFulton91}]
\label{lem-2.1}
In a burn-off game on a general graph, the vertices can be fired in
any order without affecting the length or final configuration of the
game. 
\end{lem}

\noindent
Lemma~\ref{lem-2.1} has appeared in several other places, including 
\cite{bjorner1991}, \cite{Holroyd-et-al2008}, and \cite{perkins2005},
the second of these containing a particularly succinct proof. 

Because our graphs are finite and a chip is burned during every
vertex-firing, burn-off games of infinite length are
impossible. Within the general chip-firing literature, finding
non-trivial bounds for the game length has been 
tackled more than once; see, e.g., \cite{Tardos88} and
\cite{vandenHeuvel2001}. 
For our purposes, we shall need the following
elementary result.

\begin{lem}
\label{lem-4.8}
During a burn-off game that starts with a relaxed legal configuration,
no vertex fires more than once.  
\end{lem}

\noindent
In the sandpile literature, Lemma~\ref{lem-4.8} originated in
\cite{Dhar90} as elucidated in \cite{CoriRossin2000}. Before we became
aware of its earlier existence, the second author of the present work
included it in his dissertation~\cite{perkins2005} and we included a proof
in \cite{PerkinsKayll2016}.

Our last three tools concern legal configurations.  They appeared
in \cite{perkins2005}, followed by published
proofs in \cite{KayllPerkins2013}.  Likewise with Lemma~\ref{lem-4.8},
their versions in the sandpile literature predate these citations; for example,
the first tool---Lemma~\ref{lem-2.4}---follows from the correctness of
Algorithm~\ref{alg-2.6} so dates to \cite{Dhar90}.
It characterizes the
relaxed legal configurations on general graphs $G$. In its statement,
$N_G$ denotes the `earlier neighbor' set; i.e., given an ordering
$(w_1,\ldots,w_n)$ of $V(G)$, we define $N_G(w_i):=\{w_j\colon w_{i}w_{j}
\in E(G) \text{ and } j<i\}$.

\begin{lem}
\label{lem-2.4}
A relaxed configuration $C\colon V\to\mathbb{N}$ is legal if and only
if it is possible to relabel $V$ as $w_1,\ldots,w_n$ so that 
\begin{equation}
\label{eqn-PropertyP}
C(w_i) \geq |N_G(w_i)| \text{ for } 1 \leq i \leq n.
\end{equation}
\end{lem}

The following basic result establishes that containing a
legal configuration is an inherited property for graphs;
see \cite{KayllPerkins2013} for one published proof.
 
\begin{lem}
\label{lem-2.8}
For a configuration $C\colon V(G)\to\mathbb{N}$ and a subgraph $H$ of
$G$, if $C$ is legal on $G$, then $C|_{V(H)}$ is legal on $H$.
\end{lem}

We close this section by recalling an algorithm for determining the
legality of a given configuration. The version stated here is from 
\cite{KayllPerkins2013}---a published account from 
\cite{perkins2005}---but it's essentially Dhar's `Burning
Algorithm' from \cite{Dhar90}; see also \cite{CoriRossin2000}.
The proofs of Theorems~\ref{thm-R-is-S} and \ref{thm-4.9} use this
algorithm repeatedly. 

\begin{alg}
\label{alg-2.6}
\end{alg}
\begin{center}	
	\begin{tabular}{|r p{4.5in}|}
    \hline
        \rule[-0.75em]{0pt}{2em}\textsc{Input:} & a graph $G=(V,E)$
        and a chip configuration $C\colon V \to \mathbb{N}$ on $G$ \\
        \textsc{Output:} & an answer to the question  `Is $C$
        legal?' \\
   \hline
            & \\[-0.5em]
	(1) & Let $\widehat{G} = G$. \\
	(2) & If $C(v)<\deg_{\widehat{G}}(v)$ for all $v \in V(\widehat{G})$, then
            stop; output `No'.\\ 
	(3) & Choose any $v \in V(\widehat{G})$ with
            $C(v)\geq\deg_{\widehat{G}}(v)$. \\ 
	(4) & Delete $v$ and all incident edges from $\widehat{G}$ to
        create a graph $G^-$. \\
	(5) & If $V(G^-)=\varnothing$, then stop; output `Yes'.\\
	(6) & Let $\widehat{G}=G^-$ and go to step 2. \\[0.5em]
	\hline
	\end{tabular}
  \end{center}

\bigskip

\section{Enumerating relaxed legal configurations}
\label{Sec-counting-R}

Here we present our proof of (\ref{main-RS}).  For a graph $G$, recall
that the \DEF{cone} $G^*$ is obtained from $G$ by adding a new vertex
$x$ adjacent to every vertex of $G$. This derived graph is sometimes
called the `suspension' of $G$ over $x$, but we shall not use this
term.  The reader should keep in mind the special role that the symbol
`$x$' plays in this section and the next.

\begin{thm}
\label{thm-R-is-S}
The number of relaxed legal configurations on $G$ is the number of
spanning trees of $G^*$. 
\end{thm}

\begin{proof}
We may assume that $G$ is connected, for if $H_1,\ldots,H_k$ are the
components of $G$, then---once we have $|\mathcal{R}(H_i)|=\tau(H_i^*)$
for $1\leq i\leq k$ (i.e., once we have the theorem for connected
graphs)---we obtain
\[
|\mathcal{R}(G)| = \prod\sb{i=1}\sp{k}|\mathcal{R}(H_i)|=
\prod\sb{i=1}\sp{k}\tau(H_i^*)=\tau(G^*),
\]
which is the theorem for general graphs.

Given a connected graph $G$,   
we establish algorithmically injections back and forth between
$\mathcal{R}$ and the set $\mathcal{S}$ of spanning trees
of $G^*$. Define
$A\colon\mathcal{R}\to\mathcal{S}$ via Algorithm~\ref{alg-4.2} below and
$B\colon\mathcal{S}\to\mathcal{R}$ via Algorithm~\ref{alg-4.3} below. 

\begin{alg}
\label{alg-4.2}
\end{alg}
  \begin{center}	
	\begin{tabular}{|r p{4.5in}|}
    \hline
        \rule[-0.75em]{0pt}{2em}\textsc{Input:} & a connected graph
        $G=(V,E)$ with $V = \{v_1,v_2,\ldots,v_n\}$ and a
        configuration $C\in\mathcal{R}$ \\  
        \textsc{Output:} & a spanning tree $A(C)=T^*$ of $G^*$ \\
        \hline
            & \\[-0.5em]
	(0) & Let $T^*$ be the subgraph of $G^*$ with $V(T^*)=\{x\}$,
        $E(T^*)=\varnothing$. \\ 
	(1) & Let $i=1$. \\
	(2) & Let $M_1$ be the sequence (in increasing subscript
        order) of vertices $v_k$ such that $C(v_k)=\deg_G(v_k)$; let
        $\overline{M}_1 = \{x\colon x \text{ is an entry of } M_1\}$. \\ 
	(3) & For each $v_k\in\overline{M}_1$, add $v_k$ to $V(T^*)$
        and $\{x,v_k\}$ to $E(T^*)$; if $V(T^*)=V$, then stop. \\ 
	(4) & $i \mapsfrom i+1$. \\
	(5) & Let $M_i$ be the sequence (in increasing subscript
        order) of the vertices not yet included in $V(T^*)$ that are
        neighbors of vertices in $M_{i-1}$;  let
        $\overline{M}_i=\{x\colon x \text{ is an entry in } M_i\}$. \\[0.5em]  
	& For each $u \in \overline{M}_i$, execute steps (6) through
        (9): \\[0.5em] 
	(6) & For $r=1,2,\ldots,i-1$, let
        $N_r=(v_{r,1},v_{r,2},\ldots,v_{r,k_r})$ be the sequence (in
        increasing subscript order) of the $k_r$ $G$-neighbors of $u$
        that appear in $\overline{M}_r$;  let
        $\overline{N}_r=\{x\colon x \text{ is an entry in } N_r\}$. \\ 
	(7) & Let $s=\left|\bigcup_{r=1}^{i-1}{\overline{N}_r}\right|$
        and $N=(v_{\ell_1},v_{\ell_2},\ldots,v_{\ell_s})$ be the sequence
        determined by concatenating the sequences
        $N_1,N_2,\ldots,N_{i-1}$. \\ 
	(8) & If $C(u)<\deg_G(u)-s$, then delete $u$ from $M_i$ and
        $\overline{M}_i$. \\ 
	(9) & Otherwise, $C(u)=\deg_G(u)-j$ for some $j$ with
        $1\leq j\leq s$; add $u$ to $V(T^*)$ and $\{u,v_{\ell_j}\}$ to
        $E(T^*)$. \\[0.5em] 
	(10) & If $V(T^*)=V$, then stop;  otherwise, go to step (4). \\[0.5em] 
	\hline
	\end{tabular}
  \end{center}

\bigskip
\noindent
\textit{Proof that $A$ is well-defined.} Not only must we be sure that
Algorithm~\ref{alg-4.2} outputs a spanning tree $T^*$, but also we
must check that it does not halt before doing so. To establish both of
these results, we look at each step in turn. 

\textit{Step} (2). By Algorithm~\ref{alg-2.6}, we know that at least
one vertex in a legal configuration contains at least as many chips as
its degree. Thus $\overline{M}_1$ is not empty. 

\textit{Step} (3). It is clear that $T^*$ is thus far a tree; in fact,
it is a star. 

\textit{Step} (5). We must establish that $\overline{M}_i$ is nonempty
so that the ``for each $u \in \overline{M}_i$'' instruction is
not quantifying over an empty set. We proceed by induction. In the
discussion of Step (2) above, we observed that $\overline{M}_1$ is
nonempty. By construction, all vertices in $\overline{M}_1$ are
critical. Because $C$ is a legal configuration, we may apply
Algorithm~\ref{alg-2.6} to $G$ and delete all of the vertices (in any
order) in $\overline{M}_1$. 

With these statements as our base case, our induction hypothesis is in
two parts: for fixed $i > 1$, suppose that (a)
$\overline{M}_1,\overline{M}_2,\ldots,\overline{M}_{i-1}$ are
nonempty; and (b) we may apply Algorithm~\ref{alg-2.6} to $G$ and
delete the vertices in
$\overline{M}_1,\overline{M}_2,\ldots,\overline{M}_{i-1}$ without
halting. 

Let $M = \bigcup_{j=1}^{i-1}{\overline{M}_j}$. Lemma~\ref{lem-2.8}
states that the configuration on any subgraph of a graph (on which we
have a legal configuration) must itself be legal. So, if our
application of Algorithm~\ref{alg-2.6} has deleted exactly the vertices
of $M$, then at least one of the remaining vertices $u$ of $G-M$ must be
critical in $G-M$. Suppose that $u$ is not a neighbor of any vertex in
$M$. Because $u$ is critical in $G-M$, and none of its neighbors have
been deleted in our application of Algorithm~\ref{alg-2.6}, we see
that $u$ is also critical in $G$. But this places $u$ in
$\overline{M}_1$, which contradicts the choice of $u$ in $G-M$. 

Thus, we know that $u$ is a neighbor of some vertex in $M$. Now if $u$
is not a neighbor of a vertex in $\overline{M}_{i-1}$, it must be
adjacent to, say, $s \geq 1$ vertices in
$\overline{M}_1,\overline{M}_2,\ldots,\overline{M}_{i-2}$. Thus, $u$
has been considered previously by step (8) and has been deleted each
time. Therefore, $C(u)<\deg_G(u)-s$. This shows (back in our
application of Algorithm~\ref{alg-2.6}) that if we have deleted all of
the vertices in
$\overline{M}_1,\overline{M}_2,\ldots,\overline{M}_{i-1}$, including
the $s$ neighbors of $u$, then $u$ will not be critical in $G-M$. This
contradicts the fact that $u$ is critical in $G-M$, so $u$ must be a
neighbor of a vertex in $\overline{M}_{i-1}$. 

Because $u$ is critical in $G-M$, step (8) will not delete $u$ from
$\overline{M}_i$. Thus, $\overline{M}_i$ is nonempty; this fulfills
part (a) of the induction hypothesis. We claim that any vertex $w$
placed in $\overline{M}_i$ by step (5) will survive past step (8) only
if it, too, is critical in $G-M$. For $w$ to survive step (8), we
require that $C(w)\geq\deg_G(w)-s$, where $s$ is the number of
$G$-neighbors of $w$ that appear in $M$. Since $\deg_G(w)-s$ simply
equals $\deg_{G-M}(w)$, we know that $w$ is critical in $G-M$. Thus, all
vertices in $\overline{M}_i$ can be deleted as we apply
Algorithm~\ref{alg-2.6}. This fulfills part (b) of the induction
hypothesis.  

\textit{Step} (6). Step (5) ensures that these neighbors exist.

\textit{Step} (8). The argument given above for step (5) ensures
that $\overline{M}_i$ remains nonempty after all vertices of
$\overline{M}_i$ have been processed in step (8). 

\textit{Step} (9). It is impossible to create a cycle in this step
because step (5) only considers those vertices that are not yet part of
$V(T^*)$. 

\textit{Step} (10). This step ensures that $T^*$ will be a spanning 
tree of $G^*$. 

Observe that step (9) adds at least one edge to $T^*$ since
$\overline{M}_i$ remains nonempty. Once $n-1$ edges have been added to
$T^*$, step (10) will halt the algorithm. Since Algorithm~\ref{alg-4.2}
does not halt
until it outputs a spanning tree $T^*$, the function $A$ is
well-defined. \hfill\qedsymbol 

\medskip
\noindent
\textit{Proof that $A$ is an injection.} Let $C\in\mathcal{R}$ and
$C'\in\mathcal{R}$ be two distinct relaxed legal configurations on
$G$. We prove that $A$ is an injection by showing that the spanning
trees $A(C)$ and $A(C')$ must be distinct. As Algorithm~\ref{alg-4.2}
operates on $C$ and $C'$, it must encounter a vertex $v$ for which
$C(v)\neq C'(v)$. Step (8) might remove $v$ from consideration; if
this occurs for both inputs $C$ and $C'$, then we consider a future
pass of the algorithm. Because Algorithm~\ref{alg-4.2} includes every
vertex in the output before it halts, we know that eventually we will
find a vertex $v$ for which $C(v)\neq C'(v)$ that is not removed by
step (8) concurrently for both inputs $C$ and $C'$.

Now if $v$ is removed by step (8) for one input but not the other,
then step (9) will connect $v$ to a different neighbor for the two
inputs. On the other hand, suppose that $v$ is not removed by step (8) for
either input; because $C(v)\neq C'(v)$, step (9) will connect $v$
to a different neighbor for the two inputs. In either case, $A(C)$ and
$A(C')$ must be distinct, and $A$ is an injection. \hfill\qedsymbol 

\medskip

\begin{alg}
\label{alg-4.3}
\end{alg}
\begin{center}	
	\begin{tabular}{|r p{4.5in}|}
    \hline
        \rule[-0.75em]{0pt}{2em}\textsc{Input:} & a spanning tree
        $T^*$ of $G^*$ along with an ordering $V=(v_1,v_2,\ldots,v_n)$
        of the vertices in $G$ \\ 
        \textsc{Output:} & a relaxed legal configuration
        $B(T^*)=C\in\mathcal{R}$ \\  
   \hline
            & \\[-0.5em]
	(1) & Let $M_0 = (x)$. \\
        (2) & Let $m = \max_{v \in V}\{ d_{T^*}(x,v) \}$. 
         For $j=1,2,\ldots,m$, let $M_j$ be the sequence (in
         breadth-first order, breaking ties lexicographically by
         subscript) of vertices $v$ for which $d_{T^*}(x,v)=j$; let
         $\overline{M}_j=\{x\colon x\text{ is an entry of } M_j\}$. \\ 
	(3) & For each $u\in\overline{M}_1$, let $C(u)=\deg_G(u)$. \\[0.5em]
	& For $i=2,3,\ldots,m$ and for each $u\in\overline{M}_i$,
         following the ordering in $M_i$, execute steps (4) through
         (7): \\[0.5em] 
         & \begin{tabular}{r p{4 in}}
        (4) & For $r=1,2,\ldots,i-1$, let
        $N_r=(v_{r,1},v_{r,2},\ldots,v_{r,k_r})$ be the sequence (in their
        $M_r$-ordering) of the $k_r \geq 0$ $G$-neighbors of $u$ that
        appear in $M_r$; let
        $\overline{N}_r=\{x\colon x\text{ is an entry of } N_r \}$. \\ 
        (5) & Let $s=\left| \bigcup_{r=1}^{i-1}{\overline{N}_r} \right|$. \\
        (6) & Let $N=(v_{h_1},v_{h_2},\ldots,v_{h_s})$ be the
        sequence determined by concatenating the sequences
        $N_1,N_2,\ldots,N_{i-1}$. \\ 
        (7) & For some $t\in\{1,2,\ldots,s\}$, we have
        $\{v_{h_t},u\}\in E(T^*)$; let $C(u)=\deg_G(u)-t$. \\
           \end{tabular} \\
         &  \\
	\hline
	\end{tabular}
  \end{center}

\bigskip
\noindent
\textit{Proof that B is well-defined.} In step (2), we partition $V$
into the sequences $M_1,M_2,\ldots,M_m$. Step~(3) assigns chips to the
vertices in $\overline{M}_1$, while step (7) assigns chips to the
vertices in $\overline{M}_2,\ldots,\overline{M}_m$. Therefore,
Algorithm~\ref{alg-4.3} at least produces a function
$C\colon V\to\mathbb{N}$. 

Now we use Algorithm~\ref{alg-2.6} to establish that $C$ is
legal. Since $T^*$ is a spanning tree of $G^*$, we know that
$\overline{M}_1$ is nonempty (see step (3)); hence, there is at least
one vertex $u$ such that $C(u)=\deg_G(u)$. Thus
Algorithm~\ref{alg-2.6}, given $C$ as input, can delete the vertices 
in $\overline{M}_1$. This fact is the base case in an induction argument
that proves that in Algorithm~\ref{alg-2.6}, the vertices in
$\overline{M}_1,\overline{M}_2,\ldots,\overline{M}_m$ can be deleted
in the order given by this list. Suppose that this is true for
$\overline{M}_1,\overline{M}_2,\ldots,\overline{M}_{k-1}$, where
$2\leq k<m$. For any $u \in\overline{M}_k$, step (7) assigns
$C(u)=\deg_G(u)-t \geq\deg_G(u)-s$. Recall that $s$ counts the
neighbors in $G$ of $u$ that are in
$\cup_{r=1}^{i-1}{\overline{M}_r}$; in our induction hypothesis, we
have assumed that these neighbors have been deleted from $G$,
resulting, say, in a subgraph $G'$. If other vertices in
$\overline{M}_k$ have been deleted before we consider $u$, then
$\deg_{G'}(u)$ does not increase. Thus, we have
$C|_{V(G')}(u)\geq\deg_{G'}(u)$, so $u$ can be deleted by
Algorithm~\ref{alg-2.6}. \hfill\qedsymbol  

\medskip
\noindent
\textit{Proof that B is an injection.} Suppose that
$T^*_1,T^*_2 \in\mathcal{S}$ satisfy 
\begin{equation*}
  C_1 :=B(T^*_1)=B(T^*_2)=:C_2(:= C);
\end{equation*}
we show that then $T^*_1 = T^*_2$.

Write the breadth-first orderings of $V$
determined during the computation of $B(T^*_1)$ and $B(T^*_2)$ as
$(u_i)_{i=1}^n$ and $(w_i)_{i=1}^n$, respectively. To complete the
proof, we shall find it useful to establish the following lemma. 

\begin{lem}
\label{lem-4.4}
Under the hypothesis that $C_1 = C_2$, if there exists an integer
$j\geq 1$ such that $u_i = w_i$ for all $i \in \{1,\ldots,j \}$, then
the subtree $H^*_1$ of $T^*_1$ induced on $\{ x,u_1,\ldots,u_j \}$ is
identical to the subtree $H^*_2$ of $T^*_2$ induced on
$\{x,w_1,\ldots,w_j \}$. 
\end{lem}

\begin{proof}
We induct on $j$. First note that $H^*_1$, $H^*_2$ are indeed subtrees
of $T^*_1, T^*_2$, respectively, since the sequences $(u_i),(w_i)$ are
defined by breadth-first searches on these trees. It is also clear
from the definitions of $(u_i)$, $(w_i)$ that $u_1$, $w_1$ are adjacent to
$x$ in $H^*_1$, $H^*_2$, respectively. In the case where $j=1$, these
subtrees both consist of $2$-vertex trees containing the edge
$\{x,u_1\}=\{x,w_1\}$ and are therefore identical. 

Now fix $j > 1$, assume that the lemma holds for smaller instances of
$j$, and suppose that $u_i=w_i$ for all $i\in\{1,\ldots,j \}$. Let
$G^*_0$ denote the subgraph of $G^*$ induced on the common vertex set
$U:=\{x,u_1,\ldots,u_j\}$ of $H^*_1$, $H^*_2$, and let
$G_0=G^*_0 - x$. We consider four executions of
Algorithm~\ref{alg-4.3}; in each case, the input vertex ordering is
inherited from $G$. 

The first pair of executions computes
$D_1:=B(H^*_1)$ and $D_2:= B(H^*_2)$, two configurations
on $G_0$. Since $(u_i)_{i=1}^j$, $(w_i)_{i=1}^j$ are initial segments of
$(u_i)$, $(w_i)$, it is evident from Algorithm~\ref{alg-4.3} that
$D_1$, $D_2$ are obtained from $C_1$, $C_2$ by replacing $\deg_G$ in steps
(3),(7) by $\deg_{G_0}$ and restricting the resulting functions to
$U$. Since $C_1=C_2$, we have $D_1=D_2$. For $k=1,2$ and for each
vertex $u\in V(G_0)$, let $t_k(u)$ denote the value of $t$ in step
(7) as Algorithm~\ref{alg-4.3} determines $D_k(u)$; if $D_k(u)$ is
determined in step (3), we take $t_k(u):= 0$. Then 
\begin{equation}
\label{eq-4.1}
  D_k(u)=\deg_{G_0}(u)-t_k(u) \text{ for } k = 1,2 \text{ and each }
  u \in V(G_0).
\end{equation}
The second pair of executions computes $D'_1:=B(H^*_1 - u_j)$ and
$D'_2:=B(H^*_2-w_j)$, two 
configurations on $G'_0:=G_0-u_j=G_0-w_j$. For $k=1,2$ and for each
vertex $u\in V(G'_0)$, define 
$t'_k(u)$ analogously with $t_k(u)$; now we have 
\begin{equation}
\label{eq-4.2}
  D'_k(u) = \deg_{G'_0}(u)-t'_k(u) \text{ for } k = 1,2 \text{ and each }
  u \in V(G'_0).
\end{equation}
Since $(u_i)_{i=1}^j$, $(w_i)_{i=1}^j$ are respectively breadth-first
orderings of $V(H^*_1)$, $V(H^*_2)$, the sequences $(u_i)_{i=1}^{j-1}$,
$(w_i)_{i=1}^{j-1}$ are such orderings of $V(H^*_1-u_j)$,
$V(H^*_2-w_j)$. Thus, during the second pair of executions of
Algorithm~\ref{alg-4.3} described above, every sequence $M_i$ (in the
statement of the algorithm) is the same as during the first pair of
respective executions, except, in passing from the first pair to the
second, the final vertex of $M_m$ (resp.\ $u_j$, $w_j$) has been
deleted. Therefore
\begin{equation}
\label{eq-4.3}
  t'_k(u) = t_k(u)\text{ for } k = 1,2 \text{ and each } u \in V(G'_0).
\end{equation}
Because $D_1 = D_2$, the relations in (\ref{eq-4.1}) imply that
\begin{equation}
\label{eq-4.4}
  t_1(u) = t_2(u) \text{ for each } u \in V(G_0).
\end{equation}
Comparing (\ref{eq-4.4}) with (\ref{eq-4.3}), we see that
\begin{equation}
\label{eq-4.5}
  t'_1(u) = t'_2(u) \text{ for each } u \in V(G'_0).
\end{equation}
It follows from (\ref{eq-4.2}), (\ref{eq-4.5}) that
$D'_1 = D'_2$. As these are configurations on $G'_0$, whose vertex
set is $U\smallsetminus\{ u_j \}=U \smallsetminus \{ w_j \}$, the
induction hypothesis implies that $H^*_1-u_j=H^*_2-w_j$. Finally, from 
(\ref{eq-4.4}), we have $t_1(u_j)=t_2(u_j)$, and in
Algorithm~\ref{alg-4.3}, this means that the vertex $u_j=w_j$ has
the same neighbor in $H^*_1-u_j$ as in $H^*_2-w_j$. Therefore
$H^*_1 = H^*_2$. 
\end{proof}

It follows from Lemma~\ref{lem-4.4}, with $j=n$, that if $(u_i)$ and
$(w_i)$ agree entirely, then $T^*_1 = T^*_2$. Thus, it remains only to
address the case when $u_i \not = w_i$ for some 
$i \in \{ 1,\ldots,n\}$, and here we will reach a contradiction. 

First, notice that according to Algorithm~\ref{alg-4.3}, for any 
$u\in V$, we have $C(u) = \deg_G(u)$ if and only if $u$ is adjacent to
$x$ in both of $T^*_1$, $T^*_2$. Therefore, $T^*_1$, $T^*_2$ 
do not differ in their adjacencies to $x$, and the sequences
$(u_i)$, $(w_i)$ agree in their initial entries, corresponding to the
(necessarily nonempty) neighbor sets of $x$ in $T^*_1$, $T^*_2$. If
there are $\ell$ such neighbors, then $u_i = w_i$ for 
$i \in \{1,2,\ldots,\ell \}$, and we are assuming that $\ell < n$. 

Let $i_0$ denote the least $i$ such that $u_i \not = w_i$. Since 
$\ell < i_0 \leq n$, it is easy to see that Algorithm~\ref{alg-4.3} reaches
step (7) in defining $C_1(u_{i_0})$ and $C_2(w_{i_0})$. Let 
$j = i_0 - 1$, and define $H^*_1$, $H^*_2$ as in the statement of
Lemma~\ref{lem-4.4}. Since 
\begin{equation}
\label{eq-4.6}
  u_i  =  w_i  \text{ for }   i \in \{ 1,\ldots,j \},
\end{equation}
Lemma~\ref{lem-4.4} shows that $H^*_1 = H^*_2$. From (\ref{eq-4.6}),
we also see that $w_{i_0}$ does not appear in the subsequence
$(u_i)_{i=1}^j$, and $u_{i_0}$ does not appear in the subsequence
$(w_i)_{i=1}^j$. Thus, in computing $B(T^*_1)$,
Algorithm~\ref{alg-4.3} processes $u_{i_0}$ before $w_{i_0}$, while in
computing $B(T^*_2)$, Algorithm~\ref{alg-4.3} processes $u_{i_0}$ after
$w_{i_0}$. 

Now consider the instants during the two executions of
Algorithm~\ref{alg-4.3} when step (7) defines $C_1(u_{i_0})$ and 
$C_2(u_{i_0})$. In particular, for $k=1,2$, define $t_k$ as in the
proof of Lemma~\ref{lem-4.4}, so that 
\begin{equation*}
  C_k(u_{i_0}) = \deg_G(u_{i_0}) - t_k(u_{i_0})  \text{ for }  k = 1,2.
\end{equation*}
Since $C_1 = C_2$ by hypothesis, we have
\begin{equation}
\label{eq-4.7}
  t_1(u_{i_0}) = t_2(u_{i_0}).
\end{equation}

As Algorithm~\ref{alg-4.3} executes on $T^*_1$ and is processing 
$u = u_{i_0}$, denote the sequence $N$ in step~(6) by $N_1$. Likewise,
during execution on $T^*_2$ and while processing the same vertex,
denote the corresponding sequence by $N_2$. The entries of $N_1$ are
the $G$-neighbors of $u_{i_0}$ lying (strictly) closer to $x$ in
$T^*_1$ than $u_{i_0}$. Similarly, the entries of $N_2$ are the
$G$-neighbors of $u_{i_0}$ lying (strictly) closer to $x$ in $T^*_2$
than $u_{i_0}$. Since $H^*_1 = H^*_2$, the sequence $N_1$ forms an
initial segment of the sequence $N_2$. It follows from this and
(\ref{eq-4.7}) that the $T^*_1$-neighbor of $u_{i_0}$ closer to $x$
(than $u_{i_0}$) in $T^*_1$ and the $T^*_2$-neighbor of $u_{i_0}$
closer to $x$ in $T^*_2$ are the same. A similar argument shows that
the $T^*_1$- and $T^*_2$-neighbors of $w_{i_0}$ closer to $x$ (than
$w_{i_0}$) in these trees are identical. Under these conditions,
Algorithm~\ref{alg-4.3} necessarily processes $u_{i_0}$ and $w_{i_0}$
in the same order during the computations of $B(T^*_1)$,
$B(T^*_2)$. But we concluded two paragraphs earlier that this is not
the case. This contradiction shows that the case when $u_i \not = w_i$
for some $i \in \{ 1,\ldots,n \}$ is impossible and therefore completes the
proof. 
\end{proof}

\section{Counting pairs in {\boldmath$\mathcal{R}\times V$} with specified game lengths}
\label{Sec-dist-game-lengths}

We turn now to our main results, which enumerate the pairs
$(C,v)\in\mathcal{R}(G)\times V(G)$ such that seeding $C$ at $v$
results in a game of given length $\ell$. These lean
heavily on Algorithms~\ref{alg-4.2} and \ref{alg-4.3}.  We separate
the cases $\ell=0$ and $\ell>0$ because our expression in the second
case (Theorem~\ref{thm-4.9}) does not specialize to that in the first
(Proposition~\ref{prp-4.7}).

In any event, the case $\ell=0$ is substantially easier to handle than
the other, and we address it first. Throughout this section, we
continue to write $G^*$ for the cone of $G$ (joined to $G$ at $x$).
For $v\in V$, let $t_v$ denote the number of spanning trees of
$G^*-xv$. 

\begin{prp}
 \label{prp-4.7}
 The number of pairs $(C,v)$ resulting in a game of length zero is
 $\sum\sb{v\in V}t\sb{v}$.
\end{prp}

\begin{proof}
As shown in the discussion of Algorithm~\ref{alg-4.2}, an edge
$\{x,v\}$ in $T^*$ forces $v$ to be critical in the corresponding
relaxed legal configuration, whereas $v$ will specifically \emph{not}
be critical when that edge is missing from $T^*$. So by removing this
edge from $G^*$ and enumerating the spanning trees, we count the
relaxed legal configurations in which $v$ is not
critical. Now if $v$ is the seed, it will not fire, so the game
length will be zero. Conversely, seeds in length-zero games do not
fire and hence cannot be critical. Therefore, the stated sum neither
under- nor over-counts the desired pairs.
\end{proof}

Before presenting the case $\ell>0$, we need further notation.  For
$v\in V$, let $\mathcal{T}_{v,\ell}$ denote the set of subtrees of
$G$ of order $\ell$ and including $v$. For subgraphs $H$ of $G$
(typically of the form $G-T$, for $T\in \mathcal{T}_{v,\ell}$), let
$r(H)$ denote the number $|\mathcal{R}(H)|$ of relaxed legal
configurations on $H$. 

\begin{thm}
\label{thm-4.9}
The number of pairs $(C,v)$ resulting in a game of length $\ell > 0$ is
\begin{equation*}
  \sum_{v \in V}{ \sum_{T\in\mathcal{T}_{v,\ell}} {r(G-T)} }.
\end{equation*}
\end{thm}

\begin{proof}
For $v \in V$, let $\mathcal{R}_{v,\ell}$ denote the set of relaxed
legal configurations on $G$ such that if $v$ is seeded, then the
resulting burn-off game will be of length $\ell$. For 
$R_1,R_2 \in\mathcal{R}_{v,\ell}$, define the relation $\simeq$ as
follows: suppose that when $v$ is seeded in $R_1$ and $R_2$, the
vertices that fire in either game induce the same subgraph $H$ of $G$;
suppose also that $R_1|_{V(H)} = R_2|_{V(H)}$. If, and only if, both of these
conditions hold, we write $R_1 \simeq R_2$. It is clear that
$\simeq$ is an equivalence relation on $\mathcal{R}_{v,\ell}$; let
$\mathcal{Q}_{v,\ell}$ be the set of its equivalence classes in
$\mathcal{R}_{v,\ell}$. To prove Theorem~\ref{thm-4.9}, it will be
helpful to establish injections
$A\colon\mathcal{T}_{v,\ell}\to\mathcal{Q}_{v,\ell}$ and 
$B\colon\mathcal{Q}_{v,\ell}\to\mathcal{T}_{v,\ell}$.  

Define $A\colon\mathcal{T}_{v,\ell}\to\mathcal{Q}_{v,\ell}$ as
follows. Let $T \in\mathcal{T}_{v,\ell}$, and let $H$ be the subgraph
of $G$ induced on $V(T)$. Create $H'$ as follows: to each 
$u \in V(T)$, append $\deg_G(u)-\deg_H(u)$ leaves to $u$. Let $J$ be this
set of leaves. Now let $T'$ be the spanning tree of $H'$ consisting of
$T$ and $J$. Create $T^*$ by appending the vertex $x$ and the edge 
$\{x,v \}$ to $T'$. Use $T^*$ (with $H'$ as the underlying graph) as the
input in Algorithm~\ref{alg-4.3}; let $C^*$ be the output
configuration. Let $Q$ be a configuration on $G$ defined by 
$Q(v) =C^*(v)$ and $Q(u) = C^*(u)+1$ for each 
$u \in V(H)\smallsetminus v$. Let $Z$ be
any relaxed legal configuration on $G-H$. Define $Q(w) = Z(w)$ for
each $w \in V(G-H)$. Now $Q$ is a configuration on $G$. We demonstrate
below that $Q \in\mathcal{R}_{v,\ell}$; thus, we may let
$\overline{Q}$ denote the equivalence class of $Q$. Finally, let 
$A(T) = \overline{Q}$. 

\medskip
\noindent
\textbf{Claim 1.} \textit{A is well-defined.}

\medskip
\noindent
\textit{Proof of claim.} To show that $Q \in\mathcal{R}_{v,\ell}$, we
will demonstrate that: (a) $Q$ is a relaxed legal configuration on $G$;
and (b) seeding $v$ in $Q$ results in a burn-off game of length $\ell$. 

\medskip
\noindent
(a) \textit{$Q$ is a relaxed legal configuration on $G$.}

\medskip
Because $v$ is the only neighbor of $x$ in $T^*$, only $v$ is critical
in $C^*$ (see step (7) of Algorithm~\ref{alg-4.3}). As we define $Q$,
then, adding a chip to each $u \in V(T)\smallsetminus v$ does not make
any of these vertices supercritical. We choose $Z$ to be any relaxed legal
configuration on $G-T$, so none of the vertices in $V(G-T)$ are
supercritical. Therefore, $Q$ is relaxed. 

We appeal to Algorithm~\ref{alg-2.6} to demonstrate the legality of
$Q$. We defined $C^*$ using Algorithm~\ref{alg-4.3}, so $C^*$ is a
legal configuration on $H'$. Thus, if Algorithm~\ref{alg-2.6} operates
on $C^*$, it will provide a deletion sequence $S$ of $V(H')$. Since
every $w \in J$ is a leaf, each $\deg_{H'}(w) = 1$. Since only $v$ is
critical in $C^*$, we must have $C^*(w) = 0$. Without loss of
generality, then, we may permute $S$ so that $V(H)$ is processed
before $J$ and see that this new deletion sequence $S'$ also
satisfies the requirements of Algorithm~\ref{alg-2.6}. In passing from
$C^*$ to $Q$, we let $Q(v) = C^*(v)$ and $Q(u)=C^*(u)+1$ for each
$u \in V(H)\smallsetminus v$. Because $\deg_{H'}(x)=\deg_G(x)$ for
every $x \in V(H)$, Algorithm~\ref{alg-2.6} can begin to process $Q$
on $G$ in the same order found in the initial subsequence of $S'$
containing the vertices of $V(H)$. Since we extended $Q$ to $V(G-T)$
by choosing any legal configuration $Z$ on the subgraph $G-T$,
Algorithm~\ref{alg-2.6} can finish processing $Q$, thereby confirming
the legality of $Q$.  

\medskip
\noindent
(b) \textit{Seeding $v$ in $Q$ results in a game of length $\ell$.}

\medskip
We first show that each vertex in $T$ fires, and then show that none
of the vertices in $G-T$ fire. Since $T$ has $\ell$ vertices, and 
no vertex can fire twice (by Lemma~\ref{lem-4.8}), the resulting game
will be of length $\ell$. 

Clearly, $v$ can fire. For $u \in V(T)\smallsetminus v$, let
\begin{equation*}
  S_u = \{ w \in\Gamma_H(u)\colon d_{H'}(w,x)<d_{H'}(u,x) \}
\end{equation*}
and $s_u = |S_u|$. By step (7) of Algorithm~\ref{alg-4.3}, we have
\begin{equation*}
  C^*(u)\geq \deg_{H'}(u)-s_u = \deg_G(u)-s_u.
\end{equation*}
We defined $Q(u) = C^*(u) + 1$, so once the vertices in $S_u$ fire,
the number of chips on $u$ will be at least $\deg_G(u) + 1$, allowing
$u$ to fire as well. 

For $w \in V(G-T)$, let $s_w = | \Gamma_T(w) |$. In each relaxed legal
configuration $Z$ on $G-T$, we must have
$Z(w)\leq\deg_{G-T}(w)=\deg_G(w)-s_w$. Because the vertices in $T$
contribute a total of 
$s_w$ chips to $w$ once they have all fired, the number of chips on
$w$ will never exceed $\deg_G(w)$. Since we define $Q(w) = Z(w)$, we
know that $w$ will not fire when $v$ is the seed. 

We have shown that $Q$ is a relaxed legal configuration on $G$ such
that if $v$ is seeded, the resulting game will have length $\ell$; thus,
we know that $Q \in\mathcal{R}_{v,\ell}$. Hence, $A$ is
well-defined. \hfill\qedsymbol 

\medskip
\noindent
\textbf{Claim 2.} \textit{$A$ is injective.}

\medskip
\noindent
\textit{Proof of claim.} We will show that for distinct trees  
$T_1, T_2 \in \mathcal{T}_{v,\ell}$, we have $A(T_1) \neq A(T_2)$. For this
argument, we let $Q_{T_1}$, $Q_{T_2}$ denote one of the relaxed legal
configurations on $G$ that result as we find $A(T_1)$, $A(T_2)$
respectively. (Note that $A(T_1)$ does not equal $Q_{T_1}$, but rather
$\overline{Q}_{T_1}$; similarly, $A(T_2) = \overline{Q}_{T_2}$.) 

First suppose that $T_1$ and $T_2$ contain the same $\ell$
vertices. Because $T_1$ and $T_2$ share the same vertex set, we know
that $H_1$ and $H_2$ (as defined in the proof of Claim~1) are
identical. The creation of $H'_1$ (and $H'_2$) does not involve the
structure of $T_1$ (and $T_2$), so $H'_1$ and $H'_2$ are identical
as well.  Consequently, we know that $J_1 = J_2$, which implies that
what makes $T^*_1$ and $T^*_2$ distinct is the distinct structures of
$T_1$ and $T_2$. When we use $T^*_1$ and $T^*_2$ as inputs to
Algorithm~\ref{alg-4.3}, the injective nature of the algorithm implies
that $C^*_1$ and $C^*_2$ will be distinct; thus, $Q_{T_1}|_{V(T_1)}$
and $Q_{T_2}|_{V(T_2)}$ will be distinct. Because $V(T_1) = V(T_2)$,
we have $\overline{Q}_{T_1}\neq\overline{Q}_{T_2}$. Thus,
$A(T_1)\neq A(T_2)$.

Now suppose that $T_1$ and $T_2$ do not contain the same $\ell$ vertices,
and that $A(T_1)=A(T_2)=\overline{Q}$ for some
$\overline{Q}\in\mathcal{Q}_{v,\ell}$. When we showed above that $A$ is
well-defined, 
we saw that seeding $v$ in $Q$ results in a game in which precisely
the vertices in the underlying tree fire. But the original trees
$T_1$, $T_2$ considered in this case are distinct. The deterministic
nature of burn-off games (see Lemma~\ref{lem-2.1}) prohibits this
result; the same set of vertices must fire in any burn-off game played
on a given configuration with seed $v$. Thus,
$A(T_1)\neq A(T_2)$. \hfill\qedsymbol 

\medskip
Having established that $A\colon\mathcal{T}_{v,\ell}\to\mathcal{Q}_{v,\ell}$
is a well-defined injection, we turn our 
attention to showing the same is true of
$B\colon\mathcal{Q}_{v,\ell}\to\mathcal{T}_{v,\ell}$, defined as
follows. Let 
$\overline{Q}\in\mathcal{Q}_{v,\ell}$ so that $Q\in\overline{Q}$. Let $H$
denote the subgraph induced on the vertices that fire if $v$ is seeded
in $Q$.  

Because $Q \in \overline{Q}$, seeding $v$ in $Q$ results in a burn-off
game in which the vertices of $H$ fire. With
$h=|V(H)\smallsetminus v|$, let $F=(v,u_1,u_2,\ldots,u_h)$ be such a
firing sequence of $V(H)$. For $m\in\{1,\ldots,h\}$, let $d_m$ denote
the number of $H$-neighbors of $u_m$ 
that precede $u_m$ in $F$. At the time $u_m$ fires, it must contain at
least $\deg_G(u_m)+1$ chips, so 
\begin{equation*}
  Q(u_m) \geq\deg_G(u_m)+1-d_m .
\end{equation*}
This inequality is clearly equivalent to
\begin{equation*}
  Q(u_m)\geq | \Gamma_{G-H}(u_m) | + \deg_H(u_m) + 1 - d_m,
\end{equation*}
and since $\deg_H(u_m)\geq d_m$, we may subtract
$|\Gamma_{G-H}(u_m)|$ from the right side without it becoming
negative. On the left side, subtracting $| \Gamma_{G-H}(u_m) |$
amounts to removing that many chips from $u_m$. Let $Q^*_H$ denote the
configuration on $H$ that results if, for each $m\in\{1,\ldots,h\}$, we
remove $| \Gamma_{G-H}(u_m) |$ chips from $u_m$. Thus, we have 
\begin{equation*}
  Q^*_H(u_m) \geq \deg_H(u_m)+1-d_m \text{ for all } m\in\{1,\ldots,h\}.
\end{equation*}
Since $\deg_H(u_m) \geq d_m$, we may remove one additional chip
from each $w \in V(H)\smallsetminus v$. Let $Q_H$ denote the resulting
configuration on $H$, so that 
\begin{equation}
\label{eq-FiringForward}
  Q_H(u_m) \geq \deg_H(u_m)-d_m \text{ for all } m\in\{1,\ldots,h\}.
\end{equation}
Note that $v$ is the only vertex in $V(H)$ that is critical in $Q_H$.

Our intention is to input the graph $H$ and the configuration $Q_H$
into Algorithm~\ref{alg-4.2}. The algorithm requires that $H$ be
connected and that $Q_H$ be a relaxed legal configuration. Since $H$ is a
subgraph of $G$ induced on the vertices that fire during a burn-off
game, $H$ is connected. Our choice of $Q$ comes from an equivalence
class of the relation $\simeq$ on $\mathcal{R}_{v,\ell}$, so $Q$ is a
relaxed configuration on $G$. For each $u \in V(H)$, we remove
$|\Gamma_{G-H}(u)|$ chips from $u$, so $Q^*_H$ is a relaxed
configuration on $H$. In creating $Q_H$ from $Q^*_H$, we remove a chip
from each $w \in V(H)\smallsetminus v$, so $Q_H$ is a relaxed
configuration on $H$. 

Finally, we appeal to Lemma~\ref{lem-2.4} to show that $Q_H$ is a
legal configuration on $H$. Reverse the firing sequence $F$ by
relabeling $u_{h-t+1}$ as $w_t$ for $t\in\{1,\ldots,h\}$, and label $v$ as
$w_{h+1}$; let $d'_t$ denote the number of $H$-neighbors of $w_t$
that precede $w_t$ in the sequence $(w_1,w_2,\ldots,w_{h+1})$
(thus, $d_{h-t+1}=\deg_H(w_t)-d'_t$ for $1\leq t\leq h$).
From (\ref{eq-FiringForward}), we know that
each $t\in\{1,\ldots,h\}$ satisfies
\begin{eqnarray*}
  Q_H(w_t) &=&    Q_H(u_{h-t+1}) \\
           &\geq& \deg_H(u_{h-t+1})-d_{h-t+1} \\
           &=&    \deg_H(w_t)-(\deg_H(w_t)-d'_t) \\
           &=&    d'_t,
\end{eqnarray*}
which is the condition (\ref{eqn-PropertyP})
in Lemma~\ref{lem-2.4} for these vertices. It
is easy to see that the analogous inequality holds for $v$, so $Q_H$
is a legal configuration.   

We apply Algorithm~\ref{alg-4.2} with the connected graph $H$ and the
relaxed legal configuration $Q_H$ on $H$. The algorithm outputs a
spanning tree $T^*$ of $H^*$. Because $v$ is the only vertex in $V(H)$
that is critical in $Q_H$, the only vertex adjacent to the special
vertex $x$ in $H^*$ is $v$. Let $T = T^* - x$. Finally, define
$B(\overline{Q}) = T$. This tree is clearly a member of
$\mathcal{T}_{v,\ell}$, so $B$ is well-defined. 

\medskip
\noindent
\textbf{Claim 3.} \textit{$B$ is injective.}

\medskip
\noindent
\textit{Proof of claim.} We will show that for distinct $\overline{Q},
\overline{Q'}\in\mathcal{Q}_{v,\ell}$, we have
$B(\overline{Q}) \neq B(\overline{Q'})$. Let $Q$, $Q'$ be
representatives of $\overline{Q}$, $\overline{Q'}$ respectively. Let
$H$,  $H'$ denote the subgraphs induced 
on $G$ by the $\ell$ vertices that fire when $Q$, $Q'$ respectively are
seeded at $v$. 

First, we consider the case where $H = H' =: H_0$. Because
$\overline{Q}$ and $\overline{Q'}$ are distinct, we know that
$Q|_{V(H_0)} \neq Q'|_{V(H_0)}$. Therefore, $Q_{H_0}$ and $Q'_{H_0}$
will be distinct relaxed legal configurations on $H_0$. The injective
nature of Algorithm~\ref{alg-4.2} ensures that
$B(\overline{Q}) \neq B(\overline{Q'})$. 

Second, we consider the case where $H \neq H'$. When either of these
subgraphs is used as the underlying graph in an iteration of
Algorithm~\ref{alg-4.2}, the output is a spanning tree of that
subgraph (with the edge $\{ v,x \}$, which we subsequently
delete). Since $H \neq H'$, these two trees must be distinct, so
$B(\overline{Q}) \neq B(\overline{Q'})$. \hfill\qedsymbol 

\medskip
Assisted by the following claim, finally, we will be able to turn our
attention to the inner sum that appears in the statement of
Theorem~\ref{thm-4.9}. Given $T \in\mathcal{T}_{v,\ell}$, let us
denote $A(T)$ by $\overline{Q}_T$. 

\medskip
\noindent
\textbf{Claim 4.}
\textit{For each $T\in\mathcal{T}_{v,\ell}$, we have $|\overline{Q}_T|=r(G-T)$.}

\medskip
\noindent
\textit{Proof of claim.} Because $\overline{Q}_T$ is an equivalence
class of the relation $\simeq$ on $\mathcal{R}_{v,\ell}$, it collects
all relaxed legal configurations that agree on $V(H)$. Thus, two
elements of $\overline{Q}_T$ can differ only on $V(G-T)$. By
Lemma~\ref{lem-2.8}, the legality of $Q\in\overline{Q}_T$ on $G$
implies the legality of $Q|_{V(G-T)}$ on $G-T$. Hence,
$|\overline{Q}_T|\leq r(G-T)$.  Let $L$ represent any relaxed legal
configuration counted by $r(G-T)$, and use $L$ for $Z$ in the
definition of $A(T)$. This has the effect of extending $L$ to the rest
of $G$ using $Q|_{V(T)}$, which is common to all
$Q\in\overline{Q}_T$. Because we used 
$A\colon\mathcal{T}_{v,\ell}\to\mathcal{R}_{v,\ell}$ in bringing
about this extension, the resulting configuration is legal on
$G$. Since this extension is clearly injective, we have
$|\overline{Q}_T|\geq r(G-T)$. \hfill\qedsymbol 

\medskip
To complete the proof of Theorem~\ref{thm-4.9}, it suffices to show
that for each $v \in V$, the number $|\mathcal{R}_{v,\ell}|$
of relaxed legal configurations $C$ that result in a game of length
$\ell$ when seeded at $v$ equals 
$\sum_{T\in\mathcal{T}_{v,\ell}}{r(G-T)}$. 
Since both of $A$, $B$ are injections, and both of
$\mathcal{T}_{v,\ell}$, $\mathcal{Q}_{v,\ell}$ are finite sets, it follows
that $A$ is in fact a bijection. (The same is true of $B$, but we
don't use this fact.) Thus, as $T$ runs through
$\mathcal{T}_{v,\ell}$, its image $A(T)=\overline{Q}_T$ runs through
$\mathcal{Q}_{v,\ell}$, and it follows that 
\begin{equation*}
  |\mathcal{R}_{v,\ell}| = \sum_{\overline{Q}\in\mathcal{Q}_{v,\ell}}{|\overline{Q}|}
    = \sum_{T \in\mathcal{T}_{v,\ell}}{|\overline{Q}_T|}
    = \sum_{T \in\mathcal{T}_{v,\ell}}{r(G-T)},
\end{equation*}
where Claim~4 justifies the last identity.
\end{proof}

\section{Examples}
\label{Sec:example}

We first consider an example illustrating the use of 
Theorem~\ref{thm-R-is-S}, Proposition~\ref{prp-4.7}, and
Theorem~\ref{thm-4.9} in estimating the probability distribution of
game length in a long sequence of burn-off games.  Indeed,
understanding this distribution was our primary original motivation to
establish these results.

Before getting to specifics, let us recall the stochastic process we
set up in \cite{PerkinsKayll2016}.  The state space of our Markov
chain $(X_n)_{n\geq 0}$ is the set $\mathcal{R}$.  Each transition is
determined by randomly seeding a vertex and
relaxing the resulting configuration; to be precise, given
$X_{n}\in\mathcal{R}$, the next state is determined by choosing
$v\in V$ uniformly at random and taking $X_{n+1}$ to be the
relaxation of $X_{n}+\mathbf{1}_v$. For integers $m\geq 1$ and states
$C$, we denote by $N_m(C)$ the number of visits of $(X_n)$ to $C$
during the first $m$ transition epochs. In \cite{PerkinsKayll2016}, we
proved that $(X_n)$ is irreducible and---by arguing that it has a
doubly stochastic transition matrix---has a uniform stationary
distribution. Thus, we obtained the following consequence:
\begin{equation}
  \label{thm-limit-ratio}
  \Pr\left\{\lim\sb{m\to\infty}\frac{N_m(C)}{m}=\frac{1}{|\mathcal{R}|}\right\}=1 
  \text{ for all } C\in\mathcal{R}
  \text{ (irrespective of the initial state)}.
\end{equation}
So with high probability, the long-term proportion of time that
$(X_n)$ spends in any given state is equally spread across the 
states.

Now consider the graph $G$ consisting of a triangle $K_3$ with a
pendant vertex joined to one of its vertices by a single edge.
As $\tau(G^*)=40$, Theorem~\ref{thm-R-is-S} shows that there are $40$
relaxed legal configurations on $G$. Because $G$ has order four, there
are $160$ pairs $(C,v)\in\mathcal{R}\times V$. Of these, $82$ pairs
result in a game of length zero (Proposition~\ref{prp-4.7}). We know
that burn-off games on $G$ cannot have length greater than four
(Lemma~\ref{lem-4.8}). Four applications of Theorem~\ref{thm-4.9} show
that the numbers of pairs resulting in games of length one, two, three,
and four are $35$, $16$, $15$, and $12$, respectively.
Now the uniformity in both the seed choice and the state visitation
over a long game sequence 
(viz.\ (\ref{thm-limit-ratio})) justifies the probability distribution
of game lengths displayed in Table~\ref{Length-distrib}.

\begin{table}
\caption{Distribution of lengths for burn-off games on $K_3$ plus a pendant vertex}
\begin{center}
\begin{tabular}{|l|c|c|c|c|c|}
  \hline
  \rule[-0.5em]{0em}{1.5em}\textsf{game length} & 0 & 1 & 2 & 3 & 4 \\
  \hline
  \rule[-1.1em]{0em}{2.75em}\textsf{probability} &
  $\dsp{\frac{82}{160}}$ & $\dsp{\frac{35}{160}}$ &
  $\dsp{\frac{16}{160}}$ & $\dsp{\frac{15}{160}}$ & $\dsp{\frac{12}{160}}$ \\[1em]
  \cdashline{2-6}[0.5pt/1pt]
  \multicolumn{1}{|r|}{\rule[-1.1em]{0em}{2.75em}\textsf{\small(as percent)}} &
  $51.25$ & $21.875$ & $10$ & $9.375$ & $7.5$ \\
  \hline
\end{tabular}
\end{center}
 \label{Length-distrib}
\end{table}

For comparison, we ran a computer simulation of 10,000 burn-off
games on $G$ and plotted the results together with the probabilities
in Table~\ref{Length-distrib}. This plot appears in
Figure~\ref{compare-sim-vs-anal}, where the left bars display the
simulation data and the right bars display the distribution.
We confirmed the close visual agreement between the analytical and
simulated data using a $\chi^2$ goodness-of-fit test
(more to check our simulation than our theorems!). Even with the level
of significance $\alpha$ as high as $0.1$, this test did not reject
the hypothesis that the analytical results correctly model the
simulated data.

\begin{figure}
\begin{center}
\includegraphics[width=5.65in]{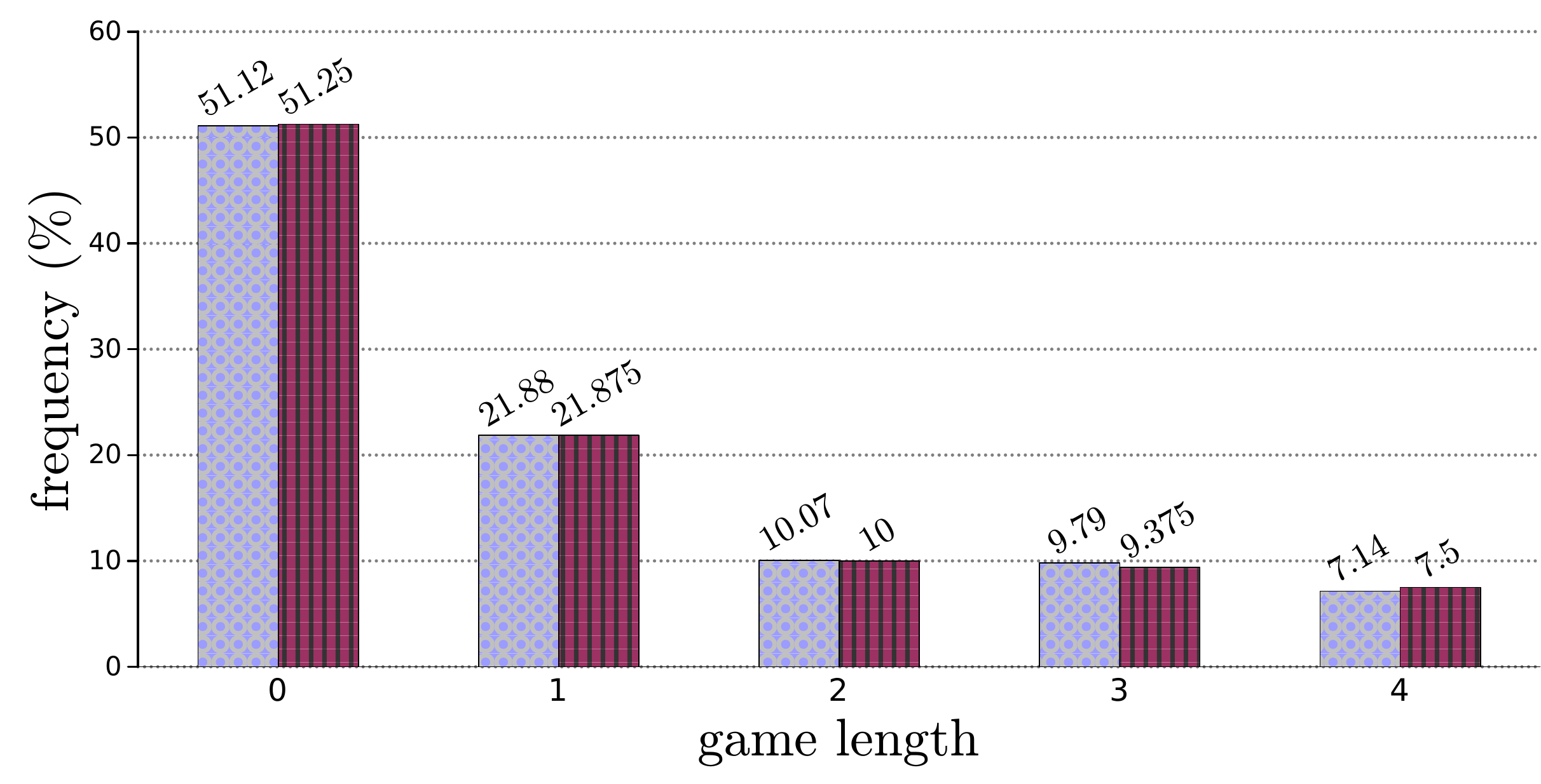}
\caption{Comparison of simulated data (10,000 trial games) with
  analytic results from Table~\protect{\ref{Length-distrib}}}
\label{compare-sim-vs-anal}
\end{center}
\end{figure}

In a follow-up paper to the present one---which has already appeared as
\cite{PerkinsKayll2017}---we apply Proposition~\ref{prp-4.7} and
Theorem~\ref{thm-4.9} to determine the game-length distribution in a
long sequence of burn-off games on a complete graph.   Thus we recover
the corresponding enumeration results obtained by Cori, Dartois, and Rossin
in \cite{CoriDartoisRossin2004}. These authors' approach is through the
(univariate) `avalanche polynomial', which is, in our terminology, a
generating function for the number of games of varying lengths. More recently, 
these polynomials were refined to their multivariate analogues in
\cite{ACFPK-avalanche2018}, where they are characterized for some
basic graph families (trees, cycles, wheels, and complete graphs).

\section{Concluding remarks}

Early papers (e.g., \cite{baktang1989}, \cite{baktangwiesenfeld1987},
\cite{tangbak88}) that inspired the invention of the abelian sandpile
model by Dhar~\cite{Dhar90} studied chip-firing games, in part,
through computer simulations. Our first example in Section~\ref{Sec:example}
is intended to illustrate how our main
results (Theorem~\ref{thm-R-is-S}, Proposition~\ref{prp-4.7},
Theorem~\ref{thm-4.9}) offer an analytic explanation for the
game-length distribution of a burn-off game, at least on the graph
considered there. Though the two results from
Section~\ref{Sec-dist-game-lengths} do not offer closed-form
expressions for the quantities being counted, the Matrix-Tree Theorem
(see, e.g., \cite{BondyMurty08}), together with
Theorem~\ref{thm-R-is-S}, render as manageable the summands
$t_v$ and $r(G-T)$ in Proposition~\ref{prp-4.7} and
Theorem~\ref{thm-4.9}. Thus, in principle, the exact probability
distribution is available. 

\subsection*{Closure}

Somewhat out of sequence, this paper brings to an end our long-term
project of producing a published account of the second author's
dissertation~\cite{perkins2005}. Besides the already mentioned
articles~\cite{KayllPerkins2013} and \cite{PerkinsKayll2016}, further
results from \cite{perkins2005} appear in \cite{KayllPerkins2009} and
\cite{PerkinsKayll2017}.  As mentioned at the end of
Section~\ref{Sec:example}, the last of these cites the present paper;
this is because it was written afterwards.

\subsection*{Acknowledgements}
Most of the manuscript for this article was finalized while the first author was
on sabbatical at the University of Otago in Dunedin, New Zealand. The
author gratefully acknowledges the support of Otago's Department of
Mathematics and Statistics. Both authors thank the referees for the 
constructive suggestions (and promptness!).

\bibliographystyle{abbrv}
\bibliography{PerkinsKayll-3-biblio}

\end{document}